\documentclass[11pt]{amsart}
\usepackage[applemac]{inputenc}

\usepackage{fullpage}
\usepackage{
 ulem,
 amsfonts}
\usepackage{amsmath,esint, mathbbol,amssymb, amsthm}

\usepackage[colorlinks=true, pdfstartview=FitV, linkcolor=blue, citecolor=blue,
 urlcolor=blue]{hyperref}
\usepackage{color}

 \usepackage{mathtools}

\date{}
\definecolor{sah}{rgb}{0.66,0.33, 0.04}
\definecolor{adel4}{cmyk}{1,0,0,0}
\definecolor{adel3}{rgb}{0.66,0.33, 0.04}
\definecolor{adel1}{cmyk}{0,0.20,1,0}
\definecolor{adel2}{cmyk}{0,0.40,1,0.30}
\definecolor{adel0}{rgb}{0.99,0.60, 0.30}
\definecolor{trut}{rgb}{0.99,0.80, 0.00}
\definecolor{trus}{rgb}{0.00, 0.50, 0.00}
 \definecolor{trust}{rgb}{0.99, 0.99, 0.80}
\definecolor{MaCouleur}{rgb}{0,0.9,0.3}
\def\virgp{\raise 2pt\hbox{,}}
\def\({\left(}
\def\){\right)}

\def\<{\langle}
\def\>{\rangle}

\theoremstyle{plain}

\newtheorem{Theo}{Theorem}
\newtheorem{lemma}{Lemma}

 \newtheorem{prop}{Proposition}

 \theoremstyle{definition}

 \newtheorem*{rema}{Remark}

\newcommand{\CC}{\mathbb{C}}
\newcommand{\NN}{\mathbb{N}}
\newcommand{\T}{\mathbb{T}}

\newcommand{\RR}{{\mathbb R}}

\usepackage[textmath,displaymath,floats,graphics,delayed]{preview}

 \title[]{Bifurcation of rotating patches  from Kirchhoff vortices }

\author[T. Hmidi]{Taoufik Hmidi}
\address{IRMAR, Universit\'e de Rennes 1\\ Campus de
Beaulieu\\ 35~042 Rennes cedex\\ France}
\email{thmidi@univ-rennes1.fr}
\author[J. Mateu]{Joan Mateu}
\address{Departament de Matem\`{a}tiques\\
Universitat Aut\`{o}noma de Barcelona\\
08193 Bellaterra, Barcelona, Catalonia} \email{ mateu@mat.uab.cat}

\begin{document}
\maketitle{}
\begin{abstract}
In this paper we prove the existence of countable branches of rotating patches bifurcating from the ellipses at some implicit angular velocities.
 \end{abstract}
\section{Introduction}
 In this paper we deal with  the vortex motion for incompressible  Euler equations in two-dimensional space. The formulations velocity-vorticity is given by the nonlinear transport equation \begin{equation}\label{vorticity}
\left\{ \begin{array}{ll}
\partial_{t}\omega +v\cdot \nabla\, \omega=0,\\
 v=\nabla^\perp\Delta^{-1}\omega,\\
\omega_{|t=0} =\omega_{0},
\end{array} \right.
\end{equation}
where $\omega$ denotes the vorticity of the velocity field $v=(v^1,v^2)$ and it is  given by $\omega=\partial_1 v^2-\partial_2 v^1$. The second equation in \eqref{vorticity} is nothing but the Biot-Savart law which can be written with a singular operator as follows:  by identifying $v=(v_1,v_2)$ with $v_1+iv_2$,  we write
\begin{equation}\label{Biot}
v(t,z)=\frac{i}{2\pi} \int_\CC
\frac{\omega(t,\xi)}{\overline{z}-\overline{\xi}}\,dA
(\xi),\quad z \in \CC,
\end{equation}
with $dA$ being the planar Lebesgue measure.  Global existence of  classical solutions is a consequence of the transport structure of the vorticity equation, for more details about this subject we refer to 
\cite{BM, Ch}.
 For less regular initial data 
Yudovich proved in \cite{Y1} that the system \eqref{vorticity} admits   a unique global
solution in the weak sense when the initial vorticity $\omega_0$
lies in $L^1 \cap L^\infty$. This allows to deal rigorously with the vortex patches which are the characteristic function of bounded domains. Therefore, it follows that when $\omega_0={\chi}_{D_0}$ with $D_0$ a bounded domain then the solution of \eqref{vorticity} preserves this structure and  $\omega(t)={\chi}_{D_t},$ with $D_t=\psi(t,D_0)$ being
the image  of $D_0$ by the flow. 
In the special case where $D_0$ is the open unit disc the vorticity
is radial and thus we get a steady flow.  Another  remarkable exact solution was   discovered by Kirchhoff \cite{Kirc} who proved that  an ellipse
$D_0$ 
performs a steady rotation about its center. More precisely, if the center is assumed to be the origin then  $D_t = e^{i t \Omega}\,D_0$, where the angular
velocity $\Omega$ is determined by the semi-axes $a$ and $b$ of the
ellipse through the formula $\Omega = ab/(a+b)^2.$ These ellipses are often referred in  the literature as Kirchhoff vortices. For a proof, see for
instance \cite[p.304]{BM} and \cite [p.232]{L}. 

The existence of general class of rotating patches, called also V-states, was  discovered numerically by Deem and Zabusky \cite{DZ}. Later on, Burbea gave an analytical proof and showed 
the  existence of the 
 $m$-fold symmetric V-states for each integer $m \geq 2$ and in this countable family the case $m=2$ corresponds to the known  Kirchhoff's ellipses.
Burbea's approach consists in using some complex analysis tools combined  with the bifurcation theory.  Notice that in this framework, the rotating patches appear as a countable collection of curves bifurcating from Rankine vortices (trivial solution) at the discrete angular velocities set $\big\{\frac{m-1}{2m}, m\geq 2\big\}.$ It is extremely interesting to look at the pictures of the limiting V-states done in  \cite{WOZ}, which are the end points of each branch. The boundary develops corners at right angles. Recently, the authors studied in \cite{HMV} the boundary regularity of the V-states close to the disc and proved    that the boundaries  are  in fact  of class   $C^\infty$ and convex. More recently, Castro, C\'ordoba and G\'omez-Serrano proved in \cite{Cor22} the analyticity of the V-states close to the disc. Notice that the existence and the regularity of the V-states for more singular nonlinear transport equations arising in geophysical flows as the surface quasi-geostrophic equations has been  studied recently in \cite{Cor1,Cor22,H-H}.

Another connected subject which has been investigated very recently in a series of paper \cite{Flierl,H-H-H, H-F-M-V,HMV2}
 is the existence of doubly connected V-states (patches with one hole).
 
The main goal of this paper is to study the second  bifurcation of rotating patches from Kirchhoff ellipses corresponding to $m=2.$ This subject was first examined by Kam in \cite{Kam}, who gave numerical evidence of the existence of some branches bifurcating from the ellipses, see also \cite{Saf}. We mention that the first bifurcation occurs at the aspect ratio $3$ corresponding to the transition regime from stability to instability. In the paper \cite{Luz} of Luzzatto-Fegiz and Willimason one can find more details about the diagram for the first bifurcations and some illustrations of the limiting V-states.  Another central problem which has been studied since the work of Love \cite{Love} is the linear and nonlinear stability of the ellipses. For instance, we mention the following papers  \cite{Guo,Tang}. As to the linear stability of the m-folds symmetric V-states, it  was conducted  by    Burbea and Landau in  \cite{Landau}. However the nonlinear stability of these structures in a small neighborhood of Rankine vortices   was done by Wan in \cite{Wan}. For further numerical  discussions, see also \cite{Cerr,DR,Mit}.

 In the current paper we intend to give  an analytical proof of the bifurcation from the ellipses. Our result reads as follows.  

 \begin{Theo}\label{hs10}
Consider the family of the ellipses $\mathcal{E}: Q\in ]0,1[\mapsto \mathcal{E}_Q$ given by the parametrization
$$
\mathcal{E}_Q=\Big\{w+\frac{Q}{w}, w\in \T\Big\}.
$$
Let $$
\mathcal{S}\triangleq \Big\{Q\in]0,1[, \,\exists\,\,\,  m\geq 3,\,\,  1+Q^m-\frac{1-Q^2}{2} m=0\Big\}.
$$
Then for each $Q\in \mathcal{S}$ there exists a nontrivial curve of  rotating vortex patches bifurcating from the curve $\mathcal{E}$ at the ellipse $\mathcal{E}_Q$. Moreover the boundary of these V-states are $C^{1+\alpha}$, $\forall \alpha\in ]0,1[$.

\end{Theo}
Before giving some details about the proof, we shall give first some remarks.
\begin{rema}
In a very recent paper \cite{Cor22}, Castro, C\'ordoba and G\'omez-Serrano  proved the analyticity of the V-states close to the ellipses. Our approach sounds to be  more easy but not so deep and cannot lead to the analyticity of the V-states. Notice that  what one could expect from iterating our method is to get the regularity $C^{n+\alpha}$ for each $n\in \NN$ but the proof does not guarantee a uniform existence interval with respect to the parameter $n$.
\end{rema}
\begin{rema}
In contrast to the bifurcation from the disc where we get a collection of $m$ folds, the V-states of Theorem \ref{hs10} are in general one or two-folds. For $m$ even we can show from the proof  that the V-states are symmetric with respect to the origin.
\end{rema}
 Now we shall sketch the proof of  Theorem \ref{hs10} which is mainly based  upon the bifurcation theory via  Crandall-Rabinowitz theorem. We shall look for a parametrization   of the boundary $\partial D$ of the rotating patches  as a small perturbation of a given ellipse. This parametrization takes the form $\Phi:\mathbb{T}\to \partial D$, with $\mathbb{T}$ is the unit circle and 
 $$
 \Phi(w)= w+Q\overline{w}+ \sum_{n\geq2} a_n w^n,\quad Q\in ]0,1[,\quad a_n\in \RR.
 $$
 Observe that when all the coefficients $a_n$ vanish then this parametrization corresponds to an ellipse, where $Q=\frac{a-b}{a+b}, $ with $a$ and $b$ being the major axis and the minor axis, respectively. As we shall see in the next section, the function $\Phi$ satisfies the nonlinear equation 
 \begin{equation*}
G(\Omega, \Phi(w))\triangleq \textnormal{Im}\Bigg\{ \bigg(2\Omega \,\overline{\Phi(w)}+\fint_{\mathbb{T}}\frac{\overline{\Phi(\xi)}-\overline{\Phi(w)}}{\Phi(\xi)-\Phi(w)}\,\Phi^\prime(\xi)d\xi\bigg)\,w\,\Phi^\prime(w)\Bigg\}=0,\quad \forall w\in \mathbb{T}.
\end{equation*}
 Setting $\alpha_Q(w)=w+Q\overline{w}$ then we retrieve the Kirchhoff solutions, meaning that,
 $$
 G(\small{\frac{1-Q^2}{4}}, \alpha_Q(w))=0,\quad \forall w\in \mathbb{T}.
 $$
 Now we introduce the function 
 \begin{equation*}
F(Q,f(w))=\textnormal{Im}\Big\{ G\big({\scriptstyle{\frac{1-Q^2}{4}}}, \alpha_Q(w)+f(w)\big)\Big\}.
\end{equation*}
then by this transformation the ellipses lead to a family of  trivial solutions: $F(Q,0)=0, \forall Q\in ]0,1[$. Therefore it is legitimate at this stage  to look for  non trivial solutions by using the bifurcation techniques in the spirit of Burbea's work \cite{B}. As we shall see, the computations of the linearized \mbox{operator $\mathcal{L}_Q\triangleq \partial_f F(Q,0)$} are a little bit more involved that the radial case but they can still be done in an explicit way. We shall see from this part that the  dispersion set $\mathcal{S}$ introduced in \mbox{Theorem \ref{hs10}} corresponds in fact  to the values of $Q$ such that the kernel of the operator  $\mathcal{L}_Q$ is one -dimensional. We shall also check that all the assumptions of Crandall-Rabinowitz theorem are satisfied and therefore the proof of the main result will follow immediately.

{\bf {Notation.}}
We need to fix some notation that will  be frequently used along this paper.
 We denote by C any positive constant that may change from line to line.
 We denote by $\mathbb{D}$ the unit disc and its boundary, the unit circle, is denoted by  $\mathbb{T}$. 
  Let $f:\mathbb{T}\to \CC$ be a continuous function, we define  its  mean value by,
$$
\fint_{\mathbb{T}} f(\tau)d\tau\triangleq \frac{1}{2i\pi}\int_{\mathbb{T}}  f(\tau)d\tau,
$$
where $d\tau$ stands for the complex integration.

  Let $X$ and $Y$ be two normed spaces. We denote by $\mathcal{L}(X,Y)$ the space of  all continuous linear maps $T: X\to Y$ endowed with its usual strong topology. 
We shall denote by $N(T)$ and $R(T)$  the kernel and the range of $T$, respectively. Finally, if $F$ is a subspace of $Y$, then $Y/ F$ denotes the quotient space.

\section{Formulation of the problem}
Following \cite{HMV2, H-F-M-V} one can see that the boundary of any smooth V-states $\chi_{D_t}$, with $D_t=e^{it\Omega} D$ is subject to the equation

\begin{equation}\label{rotsq1}
\textnormal{Re}\Bigg\{ \bigg(2\Omega \,\overline{z}+\frac{1}{2\pi i}\int_{\partial D_0}\frac{\overline{\zeta}-\overline{z}}{\zeta-z}d\zeta\bigg)\,z^\prime\Bigg\}=0,\quad \forall z\in \partial D.
\end{equation}

Recall that a curve $\gamma$ of the complex plane $\CC$ is said a regular Jordan curve if it admits a parametrization $\Phi :\mathbb{T}\to \gamma$ which  is simple  and of class $C^1$ such \mbox{that  $\Phi^\prime(w)\neq0, \quad\forall w\in \mathbb{T}$}. Note that in this case the curve $\gamma$ encloses a simply connected domain.
Now to solve the equation \eqref{rotsq1} we shall restrict ourselves to  domains whose boundaries are parametrized by a regular Jordan curve $\Phi:\mathbb{T}\to \CC$. A tangent vector to the boundary at the point $\Phi(w)$  is determined by $z^\prime=iw \Phi^\prime(w)$ and therefore \eqref{rotsq1} becomes

\begin{equation}\label{rotsq12}
\textnormal{Im}\Bigg\{ \bigg(2\Omega \,\overline{\Phi(w)}+\fint_{\mathbb{T}}\frac{\overline{\Phi(\xi)}-\overline{\Phi(w)}}{\Phi(\xi)-\Phi(w)}\,\Phi^\prime(\xi)d\xi\bigg)\,w\,\Phi^\prime(w)\Bigg\}=0,\quad \forall w\in \mathbb{T}.
\end{equation}
We shall define the object $G$ by
\begin{equation}\label{Eqzz1}
{G}(\Omega, \Phi(w))\triangleq \bigg(2\Omega \,\overline{\Phi(w)}+\fint_{\mathbb{T}}\frac{\overline{\Phi(\xi)}-\overline{\Phi(w)}}{\Phi(\xi)-\Phi(w)}\,\Phi^\prime(\xi)d\xi\bigg)\,w\,\Phi^\prime(w).
\end{equation}
It is easily seen   that the equation \eqref{rotsq12} is invariant by rotation and dilation.
Moreover, one can deduce   from this formulation Kirchhoff's result which states that an ellipse of the semi-axes $a$ and $b$ rotates with the angular velocity $\Omega=\frac{ab}{(a+b)^2}.$ Indeed, 
note that in this case the ellipse may be   parametrized  by the  conformal parametrization, 
$$
\Phi(w)=\frac{a+b}{2}\Big(w+\frac{Q}{w}\Big),\quad Q=\frac{a-b}{a+b}\cdot
$$
In the sequel we shall use the notation
$$
\alpha_Q(w)\triangleq w+\frac{Q}{w}, \quad w\in \mathbb{T}. 
$$
By  straightforward computations we get
$$
\overline{\alpha_Q(w)}\, w\,\alpha_Q^\prime(w)=1-Q^2+Q(w^2-\overline{w}^2).
$$
Using residue theorem and taking   $r>1$ we get
\begin{eqnarray*}
\fint_{\mathbb{T}}\frac{\overline{\alpha_Q(\xi)}-\overline{\alpha_Q(w)}}{\alpha_Q(\xi)-\alpha_Q(w)}\,\alpha_Q^\prime(\xi)d\xi&=&Q\fint_{r\mathbb{T}}\frac{\xi}{\alpha_Q(\xi)-\alpha_Q(w)}\alpha_Q^\prime(\xi)d\xi-\overline{\alpha_Q(w)}\\
&=&Q\alpha_Q(w)-\alpha_Q(\overline{w})\\
&=&\frac{Q^2-1}{w}\cdot
\end{eqnarray*}
It follows that
$$
\big(\frac{2}{a+b}\big)^2G(\Omega, \alpha_Q(w))=2\Omega\,Q(w^2-\overline{w}^2)+(Q^2-1)\Big(1-2\Omega-{Q}\overline{w}^2\Big).
$$
Thus 
$$
\big(\frac{2}{a+b}\big)^2\textnormal{Im}\{G(\Omega, \alpha_Q(w))\}=Q\big(4\Omega+Q^2-1\big)\textnormal{Im}(w^2)
$$
and consequently \eqref{rotsq12} is satisfied provided that 
\begin{equation*}
\Omega=\frac{1-Q^2}{4}=\frac{ab}{(a+b)^2}\cdot
\end{equation*}
This can be written in the form
\begin{equation}\label{ellip-sol}
\textnormal{Im }\big\{G\big({\scriptstyle{\frac{1-Q^2}{4}}}, \alpha_Q(w)\big)\big\}=0, \quad \forall \,w\in \mathbb{T}.
\end{equation}
Now we shall introduce the function
\begin{equation}\label{intro1}
F(Q,f(w))=\textnormal{Im}\Big\{ G\big({\scriptstyle{\frac{1-Q^2}{4}}}, \alpha_Q(w)+f(w)\big)\Big\}.
\end{equation}
From the preceding discussion we readily get
\begin{equation}\label{intro2}
F(Q,0)=0,\quad\forall Q\in (0,1).
\end{equation}
To prove Theorem \ref{hs10} we need to show the existence of nontrivial solutions of  the equation defining the V-states :
$$
  F(Q,f(w))=0,\quad\forall\, w\in \mathbb{T}.
$$
It will be done using the bifurcation theory through Crandall-Rabinowitz theorem \cite{CR}. For the completeness of the paper we recall this basic theorem and it will  referred to as sometimes by C-R theorem.
\begin{Theo}\label{CR-th} Let $X, Y$ be two Banach spaces, $V$ a neighborhood of $0$ in $X$ and let 
$
F : \RR \times V \to Y
$
with the following  properties:
\begin{enumerate}
\item $F (\lambda , 0) = 0$ for any $\lambda\in \RR$.
\item $F$ is $C^1$ and $F_{\lambda x}$ exists and are continuous.
\item $N(\mathcal{L}_0)$ and $Y/R(\mathcal{L}_0)$ are one-dimensional. 
\item {\it Transversality assumption}: $F_{\lambda x}(0, 0)x_0 \not\in R(\mathcal{L}_0)$, where
$$
N(\mathcal{L}_0) = span\{x_0\}, \quad \mathcal{L}_0\triangleq \partial_x F(0,0).
$$
\end{enumerate}
If $Z$ is any complement of $N(\mathcal{L}_0)$ in $X$, then there is a neighborhood $U$ of $(0,0)$ in $\RR \times X$, an interval $(-a,a)$ and continuous functions $\varphi: (-a,a) \to \RR$, $\psi: (-a,a) \to Z$ such that $\varphi(0) = 0$, $\psi(0) = 0$ and
$$
F^{-1}(0)\cap U=\Big\{\big(\varphi(\varepsilon), \varepsilon x_0+\xi\psi(\varepsilon)\big)\,;\,\vert \varepsilon\vert<a\Big\}\cup\Big\{(\lambda,0)\,;\, (\lambda,0)\in U\Big\}.
$$
\end{Theo}
Now we shall give a precise statement of Theorem \ref{hs10}. For this purpose we should fix the spaces $X$ and $Y$ used in C-R theorem. They are given by,
\begin{equation}\label{First1}
X=\Big\{ h\in C^{1+\alpha}(\mathbb{T}), h(w)=\sum_{n\geq2}{a_n}{w^n},\quad a_n\in \mathbb{R}\Big\}
\end{equation}
and
\begin{equation}\label{First2}
Y=\Big\{g\in C^{\alpha}(\mathbb{T}),\, g(w)=\sum_{n\geq 1}g_n\,e_n, g_n\in \RR,\, w\in\mathbb{T}\Big\}, \quad e_n(w) \triangleq \hbox{Im} (w^n).
\end{equation}

\begin{Theo}\label{hs1}
Consider the family of ellipses $\mathcal{E}: Q\in (0,1)\mapsto \mathcal{E}_Q$ given by the parametrization
$$
\mathcal{E}_Q=\Big\{w+\frac{Q}{w}, w\in \T\Big\}.
$$
Let $$
\mathcal{S}\triangleq \Big\{Q\in]0,1[, \,\exists\,\,\,  m\geq 3,\,\,  1+Q^m-\frac{1-Q^2}{2} m=0\Big\}.
$$
Then for each $Q=Q_m	\in \mathcal{S}$ there exists a nontrivial curve of  rotating vortex patches bifurcating from the curve $\mathcal{E}$ at the ellipse $\mathcal{E}_Q$. Moreover the boundary of these V-states are $C^{1+\alpha}$.

More precisely, let $Z_m$ any complement of the vector $v_m=\frac{w^{m+1}}{1-Qw^2}$ in the space $X.$  Then there exist $a
> 0$ and continuous functions  $Q : (-a,a)\to
\mathbf{R}$, $\psi: (-a,a)\to Z_m$ satisfying
$Q(0)= Q_m$, $\psi(0)=0$, such that the bifurcating curve at this point is described by,
$$
F(Q(\varepsilon), \varepsilon \frac{w^{m+1}}{1-Qw^2}+\varepsilon \psi(\varepsilon))=0.
$$ 
In particular the boundary of the V-states rotating  is described by
$$
\gamma_\varepsilon:\T\to\CC ,\quad \gamma_\varepsilon(w)=w+\frac{Q(\varepsilon)}{w}+\varepsilon \frac{w^{m+1}}{1-Qw^2}+\varepsilon \psi(\varepsilon).
$$
\end{Theo}

The proof consists in checking all the assumptions of Theorem \ref{CR-th}. This will be done in details in the next sections.

\section{Regularity of the functional}
This section is devoted to the study  of the regularity assumptions stated in  C-R theorem.  We shall study  the nonlinear  functional $F$  defining the V-states  already seen in  \eqref{intro1}. It is given through the functional $G$ as follows,
\begin{equation*}
{G}(\Omega, \Phi(w))= \bigg(2\Omega \,\overline{\Phi(w)}+\fint_{\mathbb{T}}\frac{\overline{\Phi(\xi)}-\overline{\Phi(w)}}{\Phi(\xi)-\Phi(w)}\,\Phi^\prime(\xi)d\xi\bigg)\,w\,\Phi^\prime(w)
\end{equation*}
and 
$$
F(Q,f(w))=\textnormal{Im}\Big\{ G\big({\scriptstyle{\frac{1-Q^2}{4}}}, \alpha_Q(w)+f(w)\big)\Big\}.
$$
For  $r\in (0,1)$ we denote by $B_r$  the open ball of $X$ (this  space was  introduced in  \eqref{First1})  with center  $0$ and radius $r$,
$$
B_r=\Big\{f\in X,\quad \Vert f\Vert_{C^{1+\alpha}}< r\Big\}.
$$
We shall make use at several stages of the following lemma, for more details see \cite[p. 419]{MOV}.
\begin{lemma}\label{kernel}
Let $T$ be a singular operator defined by
$$
T\phi(w)=\int_{\T}K(\xi,w)\phi(\xi)d\xi.
$$
Assume that the kernel of the operator $T$  satisfies
\begin{enumerate}
\item
$K$ is measurable on $\T \times \T$ and
$$|K(\xi,w)| \le C_0, \quad\forall\,  \xi, w \in \T.$$

\item
For each $\xi\in \T$, $w\mapsto K(\xi,w)$ is differentiable in
$\T\backslash\{\xi\}$ and
$$
\left |\partial_w K(\xi,w) \right| \le \frac{C_0}{|w-\xi|}, \quad
\forall\,  w \in \T\backslash\{\xi\}.
$$
\end{enumerate}
Then for every $0 < \alpha < 1$
 $$
 \|T\phi\|_{\alpha}\le C_0\|\phi\|_{L^\infty}.
 $$
\end{lemma}
The main result of this section reads as follows.
\begin{prop}
Let  $\varepsilon\in (0,1)$ and  $r_\varepsilon=\frac{1-\varepsilon}{2}$, then the following holds true.
\begin{enumerate}
\item The function $F:(0,\varepsilon)\times B_{r_\varepsilon}\to Y$ is of class $C^1.$
\item The partial derivative $ \partial_{Q}\partial_fF:(0,\varepsilon)\times B_{r_\varepsilon}\to Y$ is continuous.
\end{enumerate}
\end{prop}
\begin{proof}
${\bf{(1)}}$ To get this result it suffices to prove that $\partial_QF, \partial_fF:]0,\varepsilon[\times B_{r_\varepsilon}\to Y$ exist and are continuous. We shall first compute $\partial_fF(Q,f).$ This will be done  by showing first the existence of the G\^ateaux derivative and second its continuity in the strong topology. Before dealing with this problem we should first show that the functional $F$ is well-defined. For this purpose it suffices to show that the functional $G$ sends $X$ into $C^\alpha(\mathbb{T})$ and the Fourier coefficients of $G((1-Q^2)/4,\alpha_Q+f))$ are real when f belongs to $X$. As to the second claim we follow the Arxiv version  of the \mbox{paper \cite{HMV}} and for the sake of simplicity we shall  skip the details and sketch  just  the basic ideas of the proof.  First, we write
\begin{eqnarray}
\nonumber G((1-Q^2)/4,\alpha_Q(w)+f(w))&=&\frac{1-Q^2}{2}  \big[1+Q w^2+ w\overline{f(w)}\big]\big[1-Q\overline{w}^2+f^\prime(w)\big]\\
\nonumber&+&w\Phi_f^\prime(w)\fint_{\mathbb{T}}\frac{ \overline{\Phi_f(\xi)}-\overline{\Phi_f(w)}}{\Phi_f(\xi)-\Phi_f(w)}\Phi_f^\prime(\xi) d\xi\\
&\triangleq&G_1(Q,f(w))+w\Phi_f^\prime(w)\, G_2(Q,f)
\end{eqnarray}
with the notation $\Phi_f=\alpha_Q+f.$  It is clear that $G_1$ is polynomial in the variable $Q$ and bilinear on $f$ and $f^\prime$. Therefore using  the algebra structure of $C^\alpha(\mathbb{T})$ one gets
$$
\|G_1(Q,f))\|_{\alpha}\le C\|\Phi_f\|_{\alpha}\|\Phi_f\|_{1+\alpha}.
$$
This implies in particular that $G_1:(0,1)\times X\to Y$ is of class $C^\infty$. Now we shall focus on the second part $G_2$. Fix $Q\in (0,\varepsilon)$ and put $r_\varepsilon=\frac{1-\varepsilon}{2},$ then for $f\in B_{r_\varepsilon}$ we get
\begin{equation}\label{coer}
\frac{1-\varepsilon}{2}|w-\xi|\le\big|\alpha_Q(w)+f(w)-\alpha_Q(\xi)-f(\xi)\big|,\quad \forall\, \xi,w\in \mathbb{T}.
\end{equation}
Indeed,
\begin{eqnarray*}
\big|\alpha_Q(w)-\alpha_Q(\xi)\big|&=&|w-\xi|\,|1-\frac{Q}{w\xi}|\\
&\geq&(1-Q)|\xi-w|\\
&\geq&(1-\varepsilon)|\xi-w|.
\end{eqnarray*}
We combine this with the mean-value theorem applied to $f$ which is holomorphic inside the unit disc
\begin{eqnarray*}
|f(w)-f(\xi)|&\le& \|f^\prime\|_{L^\infty}|w-\xi|\\
&\le& \frac{1-\varepsilon}{2}|w-\xi|.
\end{eqnarray*}
Now using Lemma \ref{kernel} we get that $G_2(Q,f)\in C^\alpha(\mathbb{T}).$ This concludes the fact that $F$ is well-defined. Next,
we shall prove that for $f\in X$ with $\|f\|_{1+\alpha}<r_\varepsilon$ the G\^ateaux derivative $\partial_fG_2$ exists and is continuous. Straightforward computations show that this derivative is given by: for $h\in X,$
\begin{equation}\label{Non1}
\partial_f G_2(Q,f)h(w)=\sum_{j=1}^3I_j(Q,f)h(w),
\end{equation}
with 
\begin{eqnarray*}
I_1(Q,f)h(w)&=&\fint_{\mathbb{T}}\frac{{\Phi_f}( \overline\xi)-{\Phi_f}( \overline{w})}{\Phi_f(\xi)-\Phi_f(w)}h^\prime(\xi) d\xi\\
&=&\fint_{\mathbb{T}}{K}_1(\xi,w)h^\prime(\xi)d\xi,
\end{eqnarray*}
\begin{eqnarray*}
I_2(Q,f)h(w)&=&\fint_{\mathbb{T}}\frac{ h(\overline{\xi})-h(\overline{w})}{\Phi_f(\xi)-\Phi_f(w)}\Phi_f^\prime(\xi) d\xi\\
&=&\fint_{\mathbb{T}}{K}_2(\xi,w)\Phi_f^\prime(\xi)d\xi
\end{eqnarray*}
and
\begin{eqnarray*}
I_3(Q,f)h(w)&=&-\fint_{\mathbb{T}}\frac{ \big({\Phi_f}(\overline\xi)-{\Phi_f}(\overline{w})\big)(h(\xi)-h(w))}{\big(\Phi_f(\xi)-\Phi_f(w)\big)^2}\Phi_f^\prime(\xi) d\xi\\
&=&\fint_{\mathbb{T}}{K}_3(\xi,w)\Phi_f^\prime(\xi)d\xi.
\end{eqnarray*}
Notice that we have used the fact that the Fourier coefficients of $\Phi_f$ are real and therefore $\overline{\Phi_f(w)}=\Phi_f(\overline{w})$. It is easy to check  according to \eqref{coer} that
$$
|K_1(\xi,w)|=1,\quad |\partial_wK(\xi,w)|\le C_0|w-\xi|^{-1}
$$
and thus  we deduce from Lemma \ref{kernel} that 
$$
\|I_1(Q,f)h\|_\alpha\lesssim \|h^\prime\|_{L^\infty}\lesssim \|h\|_{1+\alpha}.
$$
For the second term $I_2$ we have the following estimates for the kernel
\begin{eqnarray*}
|K_2(\xi,w)|&=&\frac{|h(\overline{\xi})-h(\overline{w})|}{|\Phi_f(\xi)-{\Phi_f}(w)|}\\
&\le&\frac{2}{1-\varepsilon}\|h^\prime\|_{L^\infty}\\
|\partial_wK_2(\xi,w)|&\le& C_0\|h^\prime\|_{L^\infty}|w-\xi|^{-1}.
\end{eqnarray*}
Once again from Lemma \ref{kernel} one gets,
$$
\|I_2(Q,f)h\|_\alpha\lesssim \|h^\prime\|_{L^\infty}\|\Phi_f^\prime\|_{L^\infty}\lesssim \|h\|_{1+\alpha}.
$$
The last term can be estimated similarly to the previous one and we get
$$
\|I_3(Q,f)h\|_\alpha\lesssim \|h^\prime\|_{L^\infty}\|\Phi_f^\prime\|_{L^\infty}\lesssim \|h\|_{1+\alpha}.
$$
Putting together the preceding estimates we get 
$$
\|\partial_fG_2(Q,f)h\|_{\alpha}\le C\|h\|_{1+\alpha}.
$$
This shows the existence of G\^ateaux derivative and now we intend to  prove the continuity of the map $f\mapsto\partial_fG_2(Q,f)$ from $X$ to $\mathcal{L}(X,Y)$. This is a consequence of the following estimate that we shall prove now: for  $f,g\in B_{r_\varepsilon}$, one has
\begin{equation}\label{Frech}
\|\partial_fG_2(Q,f)h-\partial_fG_2(Q,g)h\|_\alpha\le C\|f-g\|_{1+\alpha}\| h\|_{1+\alpha}.
\end{equation}
First we write
\begin{eqnarray*}
I_1(Q,f)h(w)-I_1(Q,g)h(w)&=&\fint_{\mathbb{T}}\frac{ \big(\Phi_f(\overline{\xi})-\Phi_f(\overline{w})\big)\big((g-f)(\xi)-(g-f)(w)\big)}{(\Phi_f(\xi)-\Phi_f(w))(\Phi_g(\xi)-\Phi_g(w))}h^\prime(\xi) d\xi\\
&+&\fint_{\mathbb{T}}\frac{ (f-g)(\overline{\xi})-(f-g)(\overline{w})}{\Phi_g(\xi)-\Phi_g(w)}h^\prime(\xi) d\xi\\
&\triangleq&\fint_{\mathbb{T}}{K}_4(\xi,w)h^\prime(\xi)d\xi.
\end{eqnarray*}
By the mean value theorem we may  check that
$$
|{K}_4(\xi,w)|\le C\|f^\prime-g^\prime\|_{L^\infty}
$$
and
$$
|\partial_wK_4(\xi,w)|\le C\|f^\prime-g^\prime\|_{L^\infty}\|\xi-w|^{-1}.
$$
Thus we obtain by using Lemma \ref{kernel}
\begin{eqnarray*}
\|I_1(Q,f)h-I_1(Q,g)h\|_{\alpha}&\le &C_0\|f^\prime-g^\prime\|_{L^\infty}\|h^\prime\|_{L^\infty}\\
&\le&C\|f-g\|_{1+\alpha}\|h\|_{1+\alpha}.
\end{eqnarray*}
To estimate $I_2(Q,f)h-I_2(Q,g)h$ we shall  use the identity
$$
\frac{\Phi_f^\prime(\xi)}{\Phi_f(\xi)-\Phi_f(w)}-\frac{\Phi_g^\prime(\xi)}{\Phi_g(\xi)-\Phi_g(w)}=\frac{f^\prime(\xi)-g^\prime(\xi)}{\Phi_f(\xi)-\Phi_f(w)}-\Phi_g^\prime(\xi)\frac{(f-g)(\xi)-(f-g)(w)\big)}{(\Phi_f(\xi)-\Phi_f(w))(\Phi_g(\xi)-\Phi_g(w))}
$$
and thus
\begin{eqnarray*}
I_2(Q,f)h(w)-I_2(Q,g)h(w)&=&\fint_{\mathbb{T}}\frac{h(\overline{\xi})-h(\overline{w})}{\Phi_f(\xi)-{\Phi_f}(w)}(f^\prime(\xi)-g^\prime(\xi))d\xi\\
&-&\fint_{\mathbb{T}}\frac{ \big(h(\overline{\xi})-h(\overline{w})\big)\big((f-g)(\xi)-(f-g)(w)\big)}{(\Phi_f(\xi)-\Phi_f(w))(\Phi_g(\xi)-\Phi_g(w))} \Phi_g^\prime(\xi)d\xi\\
&\triangleq&\fint_{\mathbb{T}}{K}_5(\xi,w)(f^\prime(\xi)-g^\prime(\xi))d\xi-\fint_{\mathbb{T}}{K}_6(\xi,w)\Phi_g^\prime(\xi)d\xi.
\end{eqnarray*}
The kernels can be estimated as follows
$$
|{K}_5(\xi,w)|\le C\|h^\prime\|_{L^\infty},|\partial_w{K}_5(\xi,w)|\le C\|h^\prime\|_{L^\infty}|w-\xi|^{-1}
$$
and
$$
|{K}_6(\xi,w)|\le C_0\|h^\prime\|_{L^\infty}\|f^\prime-g^\prime \|_{L^\infty} ,|\partial_w{K}_6(\xi,w)|\le C_0\|h^\prime\|_{L^\infty}\|f^\prime-g^\prime \|_{L^\infty}|w-\xi|^{-1}.
$$
So Lemma \ref{kernel} implies that
\begin{eqnarray*}
\|I_2(Q,f)h-I_2(Q,g)h\|_{\alpha}&\le &C_0\|f^\prime-g^\prime\|_{L^\infty}\|h^\prime\|_{L^\infty}\\
&\le&C_0\|f-g\|_{1+\alpha}\|h\|_{1+\alpha}.
\end{eqnarray*}
It remains to check the continuity of $I_3$. We write
\begin{eqnarray*}
I_3(Q,f)h(w)&-&I_3(Q,g)h(w)=-\fint_{\mathbb{T}}\frac{ \big({\Phi_f}(\overline\xi)-\Phi_f(\overline w)\big)(h(\xi)-h(w))}{\big(\Phi_f(\xi)-\Phi_f(w)\big)^2}(f-g)^\prime(\xi) d\xi\\
&-&\fint_{\mathbb{T}}\frac{ \big[(f-g)(\overline\xi)-(f-g)(\overline w)\big](h(\xi)-h(w))}{\big(\Phi_f(\xi)-\Phi_f(w)\big)^2}\Phi_g^\prime(\xi)d\xi\\
&+&\fint_{\mathbb{T}}K_7(\xi,w)\Phi_g^\prime(\xi)d\xi\\
&\triangleq&\sum_{j=1}^3I_3^j(f,g)(h)(w)
\end{eqnarray*}
with
\begin{eqnarray*}
K_7(\xi,w)&=&\frac{ \big(\overline{\Phi_g}(\xi)-\overline{\Phi_g}(w)\big)\Big(\{\Phi_g(\xi)-\Phi_g(w)\}+\{\Phi_f(\xi)-\Phi_f(w)\}\Big)}{\big(\Phi_f(\xi)-\Phi_f(w)\big)^2\big(\Phi_g(\xi)-\Phi_g(w)\big)^2}\\
&\times&\big(h(\xi)-h(w)\big)\big[(f-g)(\xi)-(f-g)w)\big].
\end{eqnarray*}
The first term can be written in  the form
$$
I_3^1(f,g)(w)=\fint_{\mathbb{T}}K_8(\xi,w)(f-g)^\prime(\xi)d\xi
$$
where the kernel $K_8$ satisfies
$$
|K_8(\xi,w)|\le C \|h^\prime\|_{L^\infty} \quad\hbox{and}\quad |\partial_w K_8(\xi,w)|\le C \|h^\prime\|_{L^\infty} |\xi-w|^{-1}.
$$
This  yields in view of Lemma \ref{kernel}
\begin{eqnarray*}
\|I_3^1(f,g)\|_{\alpha}&\le& C\|h^\prime\|_{L^\infty}\|f^\prime-g^\prime\|_{L^\infty}\\
&\le& C\|h\|_{1+\alpha}\|f-g\|_{1+\alpha}.
\end{eqnarray*}
The second term can written under the form
$$
I_3^1(f,g)(w)=\fint_{\mathbb{T}}K_9(\xi,w)\Phi_g^\prime(\xi)d\xi
$$
and the kernel $K_9$ satisfies
$$
|K_9(\xi,w)|\le C_0\|f^\prime-g^\prime\|_{L^\infty} \|h^\prime\|_{L^\infty} \quad\hbox{and}\quad |\partial_w K_8(\xi,w)|\le C_0 \|f^\prime-g^\prime\|_{L^\infty}\|h^\prime\|_{L^\infty} \frac{1}{|\xi-w|}
$$
which yields in view of Lemma \ref{kernel}
\begin{eqnarray*}
\|I_3^2(f,g)\|_{\alpha}&\le& C\|h^\prime\|_{L^\infty}\|f^\prime-g^\prime\|_{L^\infty}\|\Phi_g^\prime\|_{L^\infty}\\
&\le& C\|h\|_{1+\alpha}\|f-g\|_{1+\alpha}.
\end{eqnarray*}
For the third term we can check that the kernel $K_7$ satisfies
$$
|K_7(\xi,w)|\le C\|f^\prime-g^\prime\|_{L^\infty} \|h^\prime\|_{L^\infty} \quad\hbox{and}\quad |\partial_w K_7(\xi,w)|\le C \|f^\prime-g^\prime\|_{L^\infty}\|h^\prime\|_{L^\infty} \frac{1}{|\xi-w|}
$$
which gives according to Lemma \ref{kernel}
\begin{eqnarray*}
\|I_3^3(f,g)\|_{\alpha}&\le& C\|h^\prime\|_{L^\infty}\|f^\prime-g^\prime\|_{L^\infty}\|\Phi_g^\prime\|_{L^\infty}\\
&\le& C\|h\|_{1+\alpha}\|f-g\|_{1+\alpha}.
\end{eqnarray*}
Putting together the preceding estimates we get
\begin{eqnarray*}
\|I_3(Q,f)h-I_3(Q,g)h\|_{\alpha}&\le&C\|h\|_{1+\alpha}\|f-g\|_{1+\alpha}.
\end{eqnarray*}
This achieves the proof of \eqref{Frech} and therefore the G\^ateaux derivative is Lipschitz and thus it is  continuous on the variable $f$. Therefore we conclude at this stage that the Fréchet derivative exists and coincides with Gâteaux derivative. See \cite{D} for more information.

We shall now study the regularity of $G_2$ with respect to $Q$. This reduces to studying the regularity of $Q\mapsto G_2(Q,f)$ given by
\begin{eqnarray*}
G_2(Q,f)=\fint_{\mathbb{T}}\frac{ \{\alpha_Q(\overline{\xi}))-\alpha_Q(\overline{w})\}+f(\bar\xi)-f(\bar{w})}{\{\alpha_Q({\xi}))-\alpha_Q({w})\}+f(\xi)-f({w})}\big(\alpha_Q^\prime(\xi)+f^\prime(\xi)\big) d\xi.
\end{eqnarray*}
Easy computations yields
\begin{eqnarray*}
\partial_Q G_2(Q,f)&=&\fint_{\mathbb{T}}\frac{ \xi-w}{\Phi_f(\xi)-\Phi_f(w)}\big(\alpha_Q^\prime(\xi)+f^\prime(\xi)\big) d\xi\\
&-&\fint_{\mathbb{T}}\frac{ \Phi_f(\overline\xi)-\Phi_f(\overline{w})}{\Phi_f(\xi)-\Phi_f(w)}\frac{1}{\xi^2} d\xi\\\
&-&\fint_{\mathbb{T}}\frac{ \big( \Phi_f(\overline\xi)-\Phi_f(\overline{w})\big)(\overline{\xi-w})}{\Big( \Phi_f(\xi)-\Phi_f({w})\Big)^2}\big(\alpha_Q^\prime(\xi)+f^\prime(\xi)\big) d\xi.
\end{eqnarray*}
As before, using Lemma \ref{kernel} we get for $(Q,f)\in (0,\varepsilon)\times B_{r_\varepsilon}$
$$
\|\partial_Q G_2(Q,f)\|_{\alpha}\le C_\varepsilon.
$$
Reproducing the same analysis we get for any $k\in \NN$
$$
\|\partial_Q^kG_2(Q,f)\|_{\alpha}\le C_{k,\varepsilon}.
$$
Similarly we  obtain that  $\partial_Q G_2:  (0,\varepsilon)\times B_{r_\varepsilon}\to C^\alpha(\mathbb{T})$ exists and is  continuous. Using that $\partial_fG_2$ is continuous we deduce that $G_2: (0,\varepsilon)\times B_{r_\varepsilon}\to C^\alpha(\mathbb{T})$ is $C^1$ and it follows that 
$F: (0,\varepsilon)\times B_{r_\varepsilon}\to Y$ is also $C^1$. By induction one can show that $G_2:  (0,\varepsilon)\times B_{r_\varepsilon}\to C^\alpha(\mathbb{T})$ is in fact  $C^\infty$.
\vspace{0,3cm}

$\bf(2)$ We shall check that  $\partial_Q\partial_fG_2: (0,\varepsilon)\times B_{r_\varepsilon}\to C^\alpha(\mathbb{T})$ is continuous. According to \eqref{Non1} we obtain
$$
\partial_Q\partial_fG_2(Q,f)h(w)=\sum_{j=1}^3\partial_QI_j(Q,f)h(w).
$$
$\bullet$ Estimate of $\partial_QI_1(Q,f)h.$ From  its expression we write
$$
\partial_QI_1(Q,f)h(w)=\fint_{\mathbb{T}}\partial_QK_1(\xi,w)h^\prime(\xi)d\xi.
$$
Straightforward computations yield
\begin{eqnarray*}
\partial_QI_1(Q,f)h(w)&=&\fint_{\mathbb{T}}\frac{ \xi-w}{\Phi_f(\xi)-\Phi_f(w) }h^\prime(\xi)d\xi\\
&-&\fint_{\mathbb{T}}\frac{ \big( \Phi_f(\overline\xi)-\Phi_f(\overline{w})\big)(\overline{\xi-w})}{\Big( \Phi_f(\xi)-\Phi_f({w})\Big)^2}h^\prime(\xi) d\xi.
\end{eqnarray*}
Using Lemma \ref{kernel} we get
\begin{eqnarray*}
\|\partial_QI_1(Q,f)h\|_{\alpha}&\le& C\|h^\prime\|_{L^\infty}\\
&\le&  C\|h\|_{1+\alpha}.
\end{eqnarray*}
$\bullet$ Estimate of $\partial_QI_2(Q,f)h.$ One may write

\begin{eqnarray*}
\partial_QI_2(Q,f)h(w)&=&-\fint_{\mathbb{T}}\frac{ h(\overline\xi)-h(\overline w)}{\Phi_f(\xi)-\Phi_f(w)}\frac{1}{\xi^2} d\xi\\
&-&\fint_{\mathbb{T}}\frac{ (\bar\xi-\bar{w})\big(h(\overline\xi)-h(\overline w)\big)}{\Big(\Phi_f(\xi)-\Phi_f(w)\Big)^2}\Phi_f^\prime(\xi) d\xi.
\end{eqnarray*}
Using again Lemma \ref{kernel} we find
\begin{eqnarray*}
\|\partial_QI_2(Q,f)h\|_{\alpha}&\le& C\|h^\prime\|_{L^\infty}\\
&\le&  C\|h\|_{1+\alpha}.
\end{eqnarray*}

$\bullet$ Estimate of $\partial_QI_3(Q,f)h.$ We have
\begin{eqnarray*}
\partial_QI_3(Q,f)h(w)&=&\fint_{\mathbb{T}}\frac{ \big({\Phi_f}(\overline\xi)-{\Phi_f}(\overline{w})\big)(h(\xi)-h(w))}{\big(\Phi_f(\xi)-\Phi_f(w)\big)^2}\frac{1}{\xi^2}d\xi\\
&-&\fint_{\mathbb{T}}\frac{ \big(\xi-w\big)(h(\xi)-h(w))}{\big(\Phi_f(\xi)-\Phi_f(w)\big)^2}\Phi_f^\prime(\xi)d\xi\\
&+&2\fint_{\mathbb{T}}\frac{\big(\overline\xi-\overline{w}\big) \big({\Phi_f}(\overline\xi)-{\Phi_f}(\overline{w})\big)(h(\xi)-h(w))}{\big(\Phi_f(\xi)-\Phi_f(w)\big)^3} \Phi_f^\prime(\xi)d\xi\\
&=&\fint_{\mathbb{T}}K_9(\xi,w)\frac{1}{\xi^2}d\xi+\fint_{\mathbb{T}}K_{10}(\xi,w)\Phi_f^\prime(\xi)d\xi.
\end{eqnarray*}
We can check that
$$
|K_9(\xi,w)|+|K_{10}(\xi,w)|\le C \|h^\prime\|_{L^\infty}\quad\hbox{and}\quad |\partial_wK_9(\xi,w)|+|\partial_wK_{10}(\xi,w)|\le C \|h^\prime\|_{L^\infty}\frac{1}{|w-\xi|}
$$
and therefore we get by Lemma \ref{kernel},
$$
\|\partial_QI_3(Q,f)h\|_{\alpha}\le C\|h\|_{1+\alpha}.
$$
Finally we obtain
$$
\|\partial_QI_j(Q,f)h\|_{\alpha}\le C\|h\|_{1+\alpha}.
$$
Reproducing the same analysis we get for any $k\in \NN$
$$
\|\partial_Q^kI_j(Q,f)h\|_{\alpha}\le C(k)\|h\|_{1+\alpha}
$$
and consequently
\begin{equation}\label{T1}
\|\partial_Q^k\partial_fG_2(Q,f)h\|_{\alpha}\le C\|h\|_{1+\alpha}.
\end{equation}
On the other hand the same analysis used for proving \eqref{Frech} shows that 
$$
\forall Q\in (0,\varepsilon), f, g\in B_{r_\varepsilon},\quad \|\partial_QI_j(Q,f)h-\partial_QI_j(Q,g)h\|_{\alpha}\le C \|h\|_{1+\alpha}\|f-g\|_{1+\alpha}
$$
and thus
\begin{equation}\label{T2}
\big{\|}\partial_Q\partial_f G_2(Q,f)h-\partial_Q\partial_f G_2(Q,g)h\big{\|}_{\alpha}\le C\|h\|_{1+\alpha}\|f-g\|_{1+\alpha}.
\end{equation}
Combining \eqref{T1} for the case  $k=1$ with \eqref{T2} we conclude that $\partial_Q\partial_fG_2:(0,\varepsilon)\times B_{r_\varepsilon}\to C^\alpha(\mathbb{T})$ is continuous. This achieves the proof of the proposition.
\end{proof}

\section{Study of the linearized equation}
The main goal of this section is to study some spectral properties of the linearized operator of the functional $f\in X\mapsto F(Q,f)$ in a neighborhood of   zero. This operator is defined by
$$
\mathcal{L}_Qh(w)\triangleq\frac{d}{dt}F(Q, t\,h(w))_{|t=0},\quad  h\in X,
$$
where $F$ was defined in \eqref{intro1} and the space $X$ in \eqref{First1}.

Now we introduce the following set
$$
\mathcal{S}\triangleq \Big\{Q\in]0,1[, \,\exists\,  m\geq 3,  1+Q^m-\frac{1-Q^2}{2} m=0\Big\}.
$$
For given  $m\geq 3$ the function $f_m:Q\in(0,1)\mapsto 1+Q^m-\frac{1-Q^2}{2} m$ is strictly nondecreasing and satisfies
$$f_m(0)=1-\frac{m}{2}<0,\quad  f_m(1)=2.
$$
Consequently, there is only one $Q_m\in (0,1) $ with $f_m(Q_m)=0.$ This allows to construct a function $m\mapsto Q_m$. As the map $n\mapsto f_n(Q)$ is strictly decreasing then one can readily prove that the sequence $m\mapsto Q_m$ is strictly increasing. Moreover, it is not difficult to prove the asymptotic behavior
$$
Q_m\approx 1-\frac{\alpha}{m}, \quad m\to\infty
$$ 
where $\alpha$ is the unique solution of
$$
1+e^{-\alpha}-\alpha=0.
$$
Now we shall establish the following properties for $\mathcal{L}_Q$ which yields immediately  to Theorem \ref{hs10} and Theorem \ref{hs1} according to Crandall-Rabinowitz theorem.
\begin{prop}\label{prop-spec11}
The following assertions hold true.
\begin{enumerate}
\item Let $h(w)=\sum_{n\geq2} a_n w^n\in X$, then
$$
\mathcal{L}_Qh=\sum_{n\geq1} g_{n+1}e_n; \quad e_n(w)=\hbox{Im}(w^n),
$$
with
\begin{eqnarray*}
g_2&=&-\frac12(1+Q)^2 a_2,\\
 g_3&=&-2Q^2a_3,\\
g_{n+1}&=& \Big(\frac{1-Q^2}{2}n-1-Q^n\Big)\big(a_{n+1}-Q a_{n-1}\big), \quad \forall n\geq 3.
\end{eqnarray*}
\item The kernel of $\mathcal{L}_Q$ is nontrivial if and only if $Q=Q_m\in \mathcal{S}$ and it is a one-dimensional vector space generated by
$$
v_m(w)=\frac{w^{m+1}}{1-Qw^2}\cdot
$$ 
 \item The range of $\mathcal{L}_Q$ is of co-dimension one in $Y$ and it is  given by
 $$
 R(\mathcal{L}_Q)=\Big\{ g\in C^\alpha(\mathbb{T}), g= \sum_{n\geq1\\\atop n\neq m}g_{n+1}e_n, \quad g_n \in \mathbb{R} \Big\}.
 $$
 \item Transversality assumption: for any $Q=Q_m\in\mathcal{S},$
 $$
 \partial_Q\mathcal{L}_Q v_m\notin R(\mathcal{L}_Q).
 $$
\end{enumerate}
\end{prop}
\begin{proof}
${(\bf{1})}-{(\bf{2})}$
We shall first slightly transform the expression of $G$ seen in \eqref{Eqzz1}. We may write
$$
G(\Omega, \Phi(w))=(2\Omega-1)\overline{\Phi(w)} w\Phi^\prime(w)+ w\Phi^\prime(w)\fint_{\mathbb{T}}\frac{\overline{\Phi(\xi)} \Phi^\prime(\xi)}{\Phi(\xi)-\Phi(w)} d\xi.
$$
Notice that the last integral and all the singular integrals that will appear later in this proof are  understood as the limit from the interior in the following sense: For $f:\mathbb{T}\to \CC$ a continuous function   and $w\in \mathbb{T}$, we denote by
$$
\fint_{\mathbb{T}}\frac{f(\xi)}{\Phi(\xi)-\Phi(w)} d\xi=\lim_{z\in \mathbb{D}^\star\atop  z\mapsto w} \fint_{\mathbb{T}}\frac{f(\xi)}{\Phi(\xi)-\Phi(z)} d\xi,
$$
with $\mathbb{D}^\star$ being the interior of $\mathbb{D}.$
This definition is justified by the fact that the points in the unit disc which are located  close to the boundary $\mathbb{T}$ are sent by the map $\Phi$ inside the domain enclosed by the Jordan curve $\Phi(\mathbb{T})$, since $\Phi$ is as small perturbation of the outside conformal map of the ellipse $\alpha_Q$.
Recall also that $\Omega=\frac{1-Q^2}{4}$. Now we shall compute
$$
Lh(w)\triangleq \frac{d}{dt}G(\Omega, \alpha_Q(w)+th(w))_{|t=0}.
$$
The relation with $\mathcal{L}_Q$ is given by $\mathcal{L}_Q h=\hbox{Im } L h$. For simplicity set $\alpha:=\alpha_Q$ then performing straightforward calculations one  can check that
  $$
-L h(w)=I_1(w)-\sum_{j=2}^5 I_j(w),
 $$
with 
$$
I_1(w)=\frac{1+Q^2}{2}\Big((1+Qw^2)h^\prime(w)+(w-{Q}\overline{w})h(\overline w)\Big),
$$
$$
I_2(w)=wh^\prime(w)\fint_{\mathbb{T}}\frac{\alpha(\overline{\xi})}{\alpha(\xi)-\alpha(w)}\alpha^\prime(\xi) d\xi,
$$
$$
I_3(w)=w\alpha^\prime(w)\fint_{\mathbb{T}}\frac{\alpha^\prime(\xi)}{\alpha(\xi)-\alpha(w)} h(\overline{\xi}) d\xi,
$$
$$
I_4(w)=w\alpha^\prime(w)\fint_{\mathbb{T}}\frac{\alpha(\overline{\xi})}{\alpha(\xi)-\alpha(w)}h^\prime(\xi) d\xi
$$
and
$$
I_5(w)=-w\alpha^\prime(w)\fint_{\mathbb{T}} \alpha(\overline{\xi})\frac{h(\xi)-h(w)}{\big(\alpha(\xi)-\alpha(w)\big)^2}\alpha^\prime(\xi) d\xi.
$$
$\bullet$ {\it Computation of $I_2$}. Let $z\in \mathbb{D}$ being in a very small tubular neighborhood of $\mathbb{T}$ (such that $Q/z\in \mathbb{D}$)  then by residue theorem we get
\begin{eqnarray*}
\fint_{\mathbb{T}}\frac{\alpha(\overline{\xi})}{\alpha(\xi)-\alpha(z)}\alpha^\prime(\xi) d\xi&=&\fint_{\mathbb{T}}\frac{(1+Q\xi^2)(\xi^2-Q)}{\xi^2(\xi- z)(\xi-\frac{Q}{z})}d\xi\\
&=&\fint_{\mathbb{T}}\frac{\xi^2-Q}{\xi^2(\xi- z)(\xi-\frac{Q}{z})}d\xi+Q\fint_{\mathbb{T}}\frac{\xi^2}{(\xi- z)(\xi-\frac{Q}{z})}d\xi\\
&=& Q\fint_{\mathbb{T}}\frac{\xi^2}{(\xi- z)(\xi-\frac{Q}{z})}d\xi\\
&=&Q\Big(\frac{z^2}{z-\frac{Q}{z}}+\frac{\frac{Q^2}{z^2}}{\frac{Q}{z}-z}\Big)\\
&=&Q\big(z+\frac{Q}{z}\big).
\end{eqnarray*}
Therefore letting $z$ go to $w$ we find
\begin{equation}\label{Eq1}
I_2(w)=Qw\big(w+\frac{Q}{w}\big)h^\prime(w).
\end{equation}

$\bullet$ {\it Computation of $I_3$}.  From the explicit formula of $\alpha $ we get
\begin{eqnarray*}
\fint_{\mathbb{T}}\frac{\alpha^\prime(\xi)}{\alpha(\xi)-\alpha(z)} \overline{h}(\xi) d\xi&=&\fint_{\mathbb{T}}\frac{\xi^2-Q}{\xi(\xi-z)\big(\xi-\frac{Q}{z}\big)} \overline{h}(\xi) d\xi.
\end{eqnarray*}
Since $\overline{h}$ is analytic outside the open unit disc and is at least of order two in $\frac1\xi$ at infinity    then by  residue theorem (always with $z\in \mathbb{D}$ such that $Q/z\in \mathbb{D}$)
\begin{eqnarray*}
\fint_{\mathbb{T}}\frac{\alpha^\prime(\xi)}{\alpha(\xi)-\alpha(z)} \overline{h}(\xi) d\xi&=&\fint_{\mathbb{T}}\frac{\xi^2-Q}{\xi(\xi-z)\big(\xi-\frac{Q}{z}\big)} \overline{h}(\xi) d\xi\\
&=&0.
\end{eqnarray*}
Consequently we find
\begin{equation}\label{Eq2}
I_3(w)=0.
\end{equation}

$\bullet$ {\it Computation of $I_4$}. Let $z\in \mathbb{D}$ with $Q/z\in \mathbb{D}$, since $h^\prime$ is holomorphic inside the unit disc  then we deduce by residue theorem 
\begin{eqnarray*}
\fint_{\mathbb{T}}\frac{\overline{\alpha}(\xi)}{\alpha(\xi)-\alpha(z)}h^\prime(\xi) d\xi&=&\fint_{\mathbb{T}}\frac{1+Q\xi^2}{(\xi- z)(\xi-\frac{Q}{z})}h^\prime(\xi)d\xi\\
&=&\frac{1+Qz^2}{z-\frac{Q}{z}}h^\prime(z)+\frac{1+\frac{Q^3}{z^2}}{\frac{Q}{z}-z}h^\prime\big(\frac{Q}{z}\big).
\end{eqnarray*}
Thus we obtain
\begin{equation}\label{Eq3}
I_4(w)=(1+Q w^2) h^\prime(w)-\big(1+{Q^3}\overline{w}^2\big)h^\prime\big({Q}\overline{w}\big).
\end{equation}

$\bullet$ {\it Computation of $I_5$}. Using residue theorem as in the preceding cases we find\begin{eqnarray*}
\fint_{\mathbb{T}} \overline{\alpha}(\xi)\frac{h(\xi)-h(w)}{\big(\alpha(\xi)-\alpha(w)\big)^2}\alpha^\prime(\xi) d\xi&=& \fint_{\mathbb{T}}\frac{(1+Q\xi^2)(\xi^2-Q)}{\xi(\xi-w)^2\big(\xi-\frac{Q}{w}\big)^2} \big({h}(\xi)-h(w)\big) d\xi\\
&=&-h(w) \fint_{\mathbb{T}}\frac{(1+Q\xi^2)(\xi^2-Q)}{\xi(\xi-w)^2\big(\xi-\frac{Q}{w}\big)^2}  d\xi\\
&+& \fint_{\mathbb{T}}\frac{(1+Q\xi^2)(\xi^2-Q)}{(\xi-w)^2\big(\xi-\frac{Q}{w}\big)^2} \frac{{h}(\xi)}{\xi} d\xi\\
&=&-Qh(w)+\it{J}(w),
\end{eqnarray*}
with
$$
{\it{J}}(w)\triangleq \fint_{\mathbb{T}}\frac{(1+Q\xi^2)(\xi^2-Q)}{(\xi-w)^2\big(\xi-\frac{Q}{w}\big)^2} \frac{{h}(\xi)}{\xi} d\xi.
$$
Set
$$
F_1(\xi)=\frac{(1+Q\xi^2)(\xi^2-Q)}{\big(\xi-\frac{Q}{w}\big)^2} \frac{{h}(\xi)}{\xi}\triangleq K_1(\xi) \frac{{h}(\xi)}{\xi},
$$
and
$$F_2(\xi)=\frac{(1+Q\xi^2)(\xi^2-Q)}{\big(\xi-{w}\big)^2} \frac{{h}(\xi)}{\xi}\triangleq K_2(\xi) \frac{{h}(\xi)}{\xi}\cdot
$$
Then we get successively, 
$$
K_1^\prime(w)=\frac{2Qw^2}{w-\frac{Q}{w}}\qquad  K_2^\prime\big(\frac{Q}{w}\big)=\frac{\frac{2Q^3}{w^2}}{\frac{Q}{w}-w}
$$
which give in turn 
$$
F_1^\prime(w)=\frac{1}{w-\frac{Q}{w}}\Big\{(Qw-\frac1w)h(w)+(1+Qw^2)h^\prime(w)\Big\}$$
$$F_2^\prime(\frac{Q}{w})=\frac{1}{\frac{Q}{w}-w}\Big\{\big(\frac{Q^2}{w}-\frac{w}{Q}\big)h(\frac{Q}{w})+\big(1+\frac{Q^3}{w^2}Q\big)h^\prime\Big(\frac{Q}{w}\Big)\Big\}.
$$
Applying once again residue theorem we obtain
$$
{\it{J}}(w)=F_1^\prime(w)+F_2^\prime({Q}\overline{w}).
$$
It follows that
\begin{eqnarray}\label{Eq4}
I_5(w)=(1-Q^2)\overline{w}h(w)-(1+Qw^2)h^\prime(w)+ \big({Q^2}\overline{w}-\frac{w}{Q}\Big)h({Q}\overline{w})+\big(1+{Q^3}\overline{w}^2\big)h^\prime\big({Q}\overline{w}\big).
\end{eqnarray}
Putting together the  identities \eqref{Eq1}, \eqref{Eq2}, \eqref{Eq3} and \eqref{Eq4} one gets
$$
\sum_{j=2}^5I_j(w)={(1-Q^2)}\overline{w}h(w)+ Q(Q+w^2)h^\prime(w)+ \big({Q^2}\overline{w}-\frac{w}{Q}\Big)h({Q}\overline{w})
$$
and consequently,
\begin{eqnarray*}
-L h(w)&=&\frac{1+Q^2}{2}\Big\{(1+Qw^2)h^\prime(w)+(w-{Q}\overline{w})\overline{h}(w)\Big\}+{(Q^2-1)}{\overline w}h(w)\\
&-&Q(w^2+Q)h^\prime(w)+\big(\frac{w}{Q}-{Q^2}\overline{w}\big)h\big({Q}\overline{w}\big).
\end{eqnarray*}
This can also be written in the form
\begin{eqnarray*}
-Lh(w)&=&\Big[\frac{1-Q^2}{2}+Q\frac{Q^2-1}{2} w^2\Big]h^\prime(w)+({Q^2-1})\overline{w}h(w)\\
&+&\frac{1+Q^2}{2}\big(w-{Q}\overline{w}\big)\overline{h}(w)+\big(\frac{w}{Q}-{Q^2}\overline{w}\big)h\big({Q}\overline{w}\big)
\\
&\triangleq&L_1h(w)+L_2h(w)+{L}_3h(w)+L_4h(w).
\end{eqnarray*}
It is easy to check that
\begin{eqnarray*}
{L}_1h(w)&=&\frac{1-Q^2}{2}(1-Qw^2)\sum_{n\geq2} na_n w^{n-1}\\
&=&\frac{1-Q^2}{2}\Big({2 a_2}{w}+{3a_3}{w^2}\Big)+\sum_{n\geq 3}{c_n}{w^n},
\end{eqnarray*}
with
$$
c_n=\frac{1-Q^2}{2}\Big((n+1) a_{n+1}-Q(n-1)a_{n-1}\Big),\quad\forall  n\geq 3.
$$
The computation of ${L}_2h$ is obvious, 
\begin{eqnarray*}
{L}_2h(w)&=&({Q^2-1})\sum_{n\geq2} a_n w^{n-1}\\
&=&(Q^2-1)\Big({a_2}{w}+{a_3}{w^2}\Big)+(Q^2-1)\sum_{n\geq 3}a_{n+1}{w^n}.
\end{eqnarray*}
Therefore we get
\begin{eqnarray}\label{Eq5}
{L}_1h(w)+{L}_2h(w)&=&\frac{1-Q^2}{2}a_3{w^2}+\sum_{n\geq 3}{d_n}{w^n},
\end{eqnarray}
with
\begin{equation}\label{Eq8}
d_n=\frac{1-Q^2}{2}(n-1)\big( a_{n+1}-Qa_{n-1}\big),\quad\forall n\geq 3.
\end{equation}
As to the third term, we easily  find
\begin{eqnarray*}
{L}_3h(w)&=&\frac{1+Q^2}{2}\big(w-{Q}\overline{w}\big)\sum_{n\geq2}\frac{a_n }{w^n}\\
&=&\frac{1+Q^2}{2}\big(\frac{a_2}{w}+\frac{a_3}{w^2}\big)+\frac{1+Q^2}{2}\sum_{n\geq 3}\frac{a_{n+1}-Q a_{n-1}}{w^n}\cdot
\end{eqnarray*}
Performing the same analysis yields
\begin{eqnarray*}
{L}_4h(w)&=&\big(\frac{w}{Q}-\frac{Q^2}{w}\big)\sum_{n\geq2}\frac{a_n Q^n}{w^n}\\
&=&\sum_{n\geq2}\frac{a_n Q^{n-1}}{w^{n-1}}-Q\sum_{n\geq2}\frac{a_n Q^{n+1}}{w^{n+1}}\\
&=&\frac{Q a_2}{w}+\frac{ Q^2a_3}{w^2}+\sum_{n\geq 3}\frac{Q^n\big(a_{n+1}-Q a_{n-1}\big)}{w^n}\cdot
\end{eqnarray*}
Consequently,
\begin{eqnarray}\label{Eq6}
{L}_3h(w)+{L}_4h(w)&=&\frac{(1+Q)^2}{2}\frac{a_2}{w}+\frac{1+3Q^2}{2}\frac{a_3}{w^2}\\
\nonumber&+&\sum_{n\geq3}\frac{ \widehat{d_n}}{w^n},
\end{eqnarray}
with
\begin{eqnarray}\label{Eqz}
 \widehat{d_n}&=&\Big(\frac{1+Q^2}{2}+Q^n\Big)\big( a_{n+1}-Qa_{n-1}\big).
\end{eqnarray}
According to \eqref{Eq5} and \eqref{Eq6} we get
\begin{eqnarray*}
-\mathcal{L}_Qh(w)&=&\hbox{Im}\, {L}h(w)\\
&=&-\frac{(1+Q)^2}{2} a_2e_1(w)-2Q^2 a_3e_2(w)\\
&+&\sum_{n\geq 3}(d_n- \widehat{d_n})e_n(w).
\end{eqnarray*}
with the notation $e_n(w)=\hbox{Im}(w^n).$
It follows that 
\begin{eqnarray}\label{Eqz1}
-\nonumber\mathcal{L}_Qh(w)&=&-\frac{1}{2}(1+Q)^2 a_2\, e_1(w)-2Q^2a_3\, e_2(w)\\
&+&\sum_{n\geq 3}(d_n- \widehat{d_n})\, e_n(w).
\end{eqnarray}
Combining \eqref{Eq8} and \eqref{Eqz} we obtain
\begin{equation}\label{hm1}
d_n- \widehat{d_n}=\Big(\frac{1-Q^2}{2}n-1-Q^n\Big)\big(a_{n+1}-Q a_{n-1}\big).
\end{equation}
Therefore the equation $\hbox{Im}\, \mathcal{L}h(w)=0$ is equivalent to the linear system

\begin{eqnarray*}
(1+Q)^2 a_2&=&0\\
 Q^2a_3&=&0\\
\Big(\frac{1-Q^2}{2}n-1-Q^n\Big)\big(a_{n+1}-Q a_{n-1}\big)&=& 0, \quad \forall n\geq 3.
\end{eqnarray*}

Since $Q\in (0,1)$ then necessarily 
$$
a_2=a_3=0.
$$
The last equation is equivalent to
\begin{equation}\label{cond1}
\Big(1+Q^n-\frac{1-Q^2}{2} n\Big)\big( a_{n+1}- Qa_{n-1}\big)=0, \quad\forall n\geq 3.
\end{equation}

Let $m\geq 3$, then we know from the beginning of this section the existence of only one solution  $Q=Q_m\in(0,1)$ of the equation
$$
1+Q^{m}-\frac{1-Q^2}{2} m=0.
$$
Moreover, the left part of this equality defines a strictly decreasing function in $m$ implying that
$$
1+Q^{n}-\frac{1-Q^2}{2} n\neq0,\quad \forall n\neq m.
$$
Thus \eqref{cond1} is equivalent to
$$
a_{n+1}=Q a_{n-1},\forall n\geq 3, \quad n\neq m.
$$
Thus the dynamical system \eqref{cond1} combined with the vanishing  two first values admits the following  solutions
$$
 \forall n\geq 0, a_{m+1+2n}={Q^n} a_{m+1}\quad \hbox{and} \quad 0  \quad \hbox{otherwise}
$$
 This means that 
 the associated kernel  is one dimensional generated by the eigenfunction 
$$
v_{m}(w)=\sum_{n\geq0}{Q^n\, w^{m+1+2n}}=\frac{w^{m+1}}{1-Qw^2}\cdot
$$

${(\bf{3})}$ Let $Q\in \mathcal{S}$ and $m$ being the frequency such that $Q=Q_m$. We will show that the range $R(\mathcal{L}_{Q_m})$  coincides with the closed subspace 
$$
\mathcal{Y}\triangleq\Big\{g\in C^\alpha(\mathbb{T});\quad  g=\sum_{n\geq1\atop\\ n\neq m} g_{n+1}\,e_n, \quad g_n\in\RR\Big\}.
$$
From  \eqref{Eqz1} and \eqref{hm1} one sees that the range of $\mathcal{L}_Q$ is contained in $\mathcal{Y}$. Conversely, let $g\in \mathcal Y$ we shall look for $h(w)=\sum_{n\geq2}a_n w^n\in C^{1+\alpha}(\mathbb{T})$ such that $\mathcal{L}_Qh=g.$ Once again from \eqref{Eqz1} this is equivalent to
\begin{eqnarray*}
-\frac12(1+Q)^2 a_2&=&g_2,\\
 -2Q^2a_3&=&g_3,\\
\Big(\frac{1-Q^2}{2}n-1-Q^n\Big)\big(a_{n+1}-Q a_{n-1}\big)&=& g_{n+1}, \quad \forall n\geq 3, \quad n\neq m.
\end{eqnarray*}
This determines uniquely the sequence $(a_n)_{2\le n\le m}$ and for $n\geq m+1$ one has 
 the recursive formula
$$
a_{n+1}-Q a_{n-1}=\frac{g_{n+1}}{\frac{1-Q^2}{2}n-1-Q^n},\quad \forall n\geq m+1.
$$
The only free coefficient is $a_{m+1}$ and therefore the solutions of the above system form one-dimensional affine space. To prove that any pre-image   $h$ belongs to $C^{1+\alpha}(\mathbb{T})$ it suffices to show it for the function $H(w)=\sum_{n\geq m+2}a_n w^n.$ Set    $$
 \quad G(w)=\sum_{n\geq m+1}\frac{g_{n+1}}{\frac{1-Q^2}{2}n-1-Q^n}w^{n}, \quad R_m(w)=\sum_{n\geq m+1} g_{n+1} w^n.
$$
Then
\begin{eqnarray*}
H(w)&=&w\sum_{n\geq m+1}a_{n+1} w^n\\
&=&wQ\sum_{n\geq m+1}a_{n-1} w^n+\sum_{n\geq m+1}\frac{g_{n+1}}{\frac{1-Q^2}{2}n-1-Q^n}w^{n+1}\\
&=& w^2Q\sum_{n\geq m}a_{n} w^{n}+wG(w)\\
&=& w^2 Q H(w)+w^2 Q(a_m w^m+a_{m+1} w^{m+1})+wG(w).
\end{eqnarray*}
Therefore 
$$
H(w)=\frac{1}{1-Q w^2}\Big( w^2 Q(a_m w^m+a_{m+1} w^{m+1})+wG(w)\Big).
$$
The problem reduces then to  check that $G\in C^{1+\alpha}(\mathbb{T}).$ We split $G$  into two terms as follows
\begin{eqnarray*}
G(w)&=&\sum_{n\geq m+1}\frac{g_{n+1}}{\frac{1-Q^2}{2}n-1}w^{n}+\sum_{n\geq m+1}\frac{Q^n g_{n+1}}{(\frac{1-Q^2}{2}n-1) (\frac{1-Q^2}{2}n-1-Q^n)}w^{n}\\
&=&G_1+G_2.
\end{eqnarray*}
Since the sequence $(g_n)$ is bounded then  for large $n$ one gets
\begin{eqnarray*}
 \frac{Q^n |g_{n+1}|}{(\frac{1-Q^2}{2}n-1) (\frac{1-Q^2}{2}n-1-Q^n)}&\le C \frac{Q^n}{(1-Q^2)^2}.
\end{eqnarray*}
This shows that $G_2\in C^k(\mathbb{T})$ for all $k\in \NN$. Let us now prove that $G_1\in C^{1+\alpha}(\mathbb{T}).$ First from the embedding $C^\alpha(\mathbb{T})\hookrightarrow L^\infty(\mathbb{T})\hookrightarrow L^2(\mathbb{T})$ one obtains
\begin{eqnarray*}
\sum_{n}|g_{n+1}|^2\lesssim \|g\|_{C^\alpha}^2.
\end{eqnarray*}
Therefore by Cauchy-Scwharz
\begin{eqnarray*}
\|G_1\|_{L^\infty}&\lesssim& \sum_{n\geq m+1}\frac{|g_{n+1}|}{n}\\
&\lesssim&\Big(\sum_{n\geq m+1}|g_{n+1}|^2\Big)^{\frac12}\\
&\lesssim& \|g\|_{C^\alpha}.
\end{eqnarray*}
It remains to prove that $G_1^\prime\in C^\alpha(\mathbb{T}).$ Differentiating term by term the series we get
\begin{eqnarray*}
G_1^\prime&=&\sum_{n\geq m+1}\frac{n\,g_{n+1}}{\frac{1-Q^2}{2}n-1}w^{n-1}\\
&=& \frac{2}{1-Q^2}\sum_{n\geq m+1}g_{n+1}w^{n-1}+ \frac{2}{1-Q^2}\sum_{n\geq m+1}\frac{\,g_{n+1}}{\frac{1-Q^2}{2}n-1}w^{n-1}\\
&=& \frac{2}{1-Q^2}\frac1wR_m(w)+\frac{4}{(1-Q^2)^2}\sum_{n\geq m+1}\frac{\,g_{n+1}}{n}w^{n-1}\\
&+&\frac{4}{(1-Q^2)^2}\sum_{n\geq m+1}\frac{\,g_{n+1}}{n(\frac{1-Q^2}{2}n-1)}w^{n-1}\\
&\triangleq& G_3(w)+G_4(w)+G_5(w).
\end{eqnarray*}
The function $G_3$ is clearly in $C^\alpha(\mathbb{T})$ according to the assumption $g\in C^\alpha$. The function $G_4$ belongs to $L^\infty(\mathbb{T})$  and $G_4^\prime\in C^\alpha(\mathbb{T})$. Indeed, 
\begin{eqnarray*}
\|G_4\|_{L^\infty}&\lesssim& \sum_{n\geq p+1}\frac{|g_{n+1}|}{n}\\
&\lesssim&\Big(\sum_{n\geq m+1}|g_{n+1}|^2\Big)^{\frac12}\\
&\lesssim& \|G\|_{L^\infty}.
\end{eqnarray*}
Moreover
$$
(wG_4)^\prime=\frac{4}{(1-Q^2)^2}\sum_{n\geq m+1}g_{n+1}w^{n-1}=\frac{4}{(1-Q^2)^2}\frac1w R_m(w).
$$
This gives $(wG_4)^\prime \in C^\alpha(\mathbb{T})$ and thus $G_4\in C^{1+\alpha}(\mathbb{T})$. On the other hand
$$
wG_5^\prime(w)=\frac{4}{(1-Q^2)^2}\sum_{n\geq m+1}\frac{\,g_{n+1}}{\frac{1-Q^2}{2}n-1}w^{n-1}.
$$
Arguing as before we see that $wG_5^\prime\in L^\infty(\mathbb{T})$ and belongs also to  $C^{\alpha}(\mathbb{T})$. This shows that $G_1^\prime\in C^{\alpha}(\mathbb{T})$ which gives that  $G_1\in C^{1+\alpha}(\mathbb{T})$. This shows finally that any pre-image  of $g$ belongs to the \mbox{space $X.$}

$(\bf{4})$ Let $m\geq3$ be an integer and $Q=Q_m$ the associated element in the set $\mathcal{S}.$ We have seen that the kernel is one-dimensional generated by 
$$v_m(w)=\frac{w^{m+1}}{1-Qw^2}=\sum_{n\geq0} Q^n w^{m+1+2n}.
$$ 
We shall compute $\partial_Q\partial_f F(Q_m,0)v_m$ which coincides with $\big\{\partial_Q\mathcal{L}_{Q}v_m\big\}_{Q=Q_m}$. The transversality condition that we shall check is 
$$
\big\{\partial_Q\mathcal{L}_{Q}v_m\big\}_{Q=Q_m}\notin R(\mathcal{L}_{Q_m}).
$$
From the structure of the range of $\mathcal{L}$ this is equivalent to prove that the coefficient of $e_m$ in the decomposition $\big\{\partial_Q\mathcal{L}_{Q}v_m\big\}_{Q=Q_m}$ is not zero. From  \eqref{Eqz1} and \eqref{hm1} this coefficient is given by
\begin{eqnarray*}
\{\partial_Q(d_m-\widehat{d}_m)\}_{Q=Q_m}&=&-m(Q_m+Q_m^{m-1})(a_{m+1}-Q_ma_{m-1})+(\frac{1-Q_m^2}{2}m-1-Q_m^m)a_{m-1}\\
&=&-m(Q_m+Q_m^{m-1})+0\\
&\neq&0.
\end{eqnarray*}
We have used the fact that $a_{m+1}=1$ and $a_{m-1}=0.$
This achieves the transversality assumption and therefore the proof of Proposition \ref{prop-spec11} is complete.
\end{proof}

\begin{ackname}
We are indebted to Joan Verdera for the fruitful discussion on this subject. This work was partially supported by the grants of Generalitat de Catalunya 2014SGR7,
Ministerio de Economia y Competitividad MTM 2013-4469,
ECPP7- Marie Curie project MAnET  and the ANR project Dyficolti ANR-13-BS01-0003- 01.\end{ackname}

\end{document}